\documentclass[10pt,reqno]{amsart}

\usepackage{amssymb, amsmath, amsthm,geometry}
 \usepackage{xcolor}
\usepackage{hyperref}
\usepackage{enumerate}
\newtheorem{thm}{Theorem}[section]
\newtheorem{lem}[thm]{Lemma}
\newtheorem{cor}[thm]{Corollary}
\newtheorem{prop}[thm]{Proposition}
\newtheorem{defn}[thm]{Definition}
\newtheorem{rmk}{Remark}

\numberwithin{equation}{section}
\newcommand{\bel}{\begin{equation} \label}
\newcommand{\ee}{\end{equation}}
\def\beq{\begin{equation}}
\def\eeq{\end{equation}}
\newcommand{\bea}{\begin{eqnarray}}
\newcommand{\eea}{\end{eqnarray}}
\newcommand{\beas}{\begin{eqnarray*}}
\newcommand{\eeas}{\end{eqnarray*}}

\newcommand{\R}{\mathbb{R}}
 
\newcommand{\N}{\mathbb{N}}

\newcommand{\cO}{\mathcal{O}}

\newcommand{\M}{\mathcal{M}}

\allowdisplaybreaks

\def\epsilon{\varepsilon}
\def\phi {\varphi}

\def\I{\,|\,}
\DeclareMathOperator{\dis}{dist}
\def\p{\partial}

\renewcommand{\leq}{\leqslant}
\renewcommand{\geq}{\geqslant}

\title{A Finite Element Data Assimilation Method For The Wave Equation}
\author[Erik Burman]{Erik Burman}
\address{Department of Mathematics, University College London, London, UK-WC1E  6BT, United Kingdom}
\email{e.burman@ucl.ac.uk}
\thanks{EB acknowledges funding by EPSRC grants EP/P01576X/1 and
  EP/P012434/1}
\author[Ali Feizmohammadi]{Ali Feizmohammadi}
\address{Department of Mathematics, University College London, London, UK-WC1E  6BT, United Kingdom}
\email{a.feizmohammadi@ucl.ac.uk}
\author[Lauri Oksanen]{Lauri Oksanen}
\address{Department of Mathematics, University College London, London, UK-WC1E  6BT, United Kingdom}
\email{l.oksanen@ucl.ac.uk}
\thanks{LO acknowledges funding by EPSRC grants EP/P01593X/1 and EP/R002207/1}

\date{}
\begin{document}
\maketitle
\begin{abstract}
We design a primal-dual stabilized finite element method for the
numerical approximation of a data assimilation problem subject to the
acoustic wave equation. For the forward problem, piecewise affine, continuous, finite element
functions are used for the approximation in space and backward
differentiation is used in time. Stabilizing terms are added on the
discrete level. The design of these terms is driven by numerical
stability and the stability of the continuous problem, with the
objective of minimizing the computational error. Error estimates are then
derived that are optimal with respect to the approximation properties
of the numerical scheme and the stability properties of the continuous
problem.
The effects of discretizing the (smooth) domain boundary and other perturbations
in data are included in the analysis.
\end{abstract}

\section{Introduction}
We consider a data assimilation problem for the acoustic wave equation, formulated as follows. 
Let $n \in \{2,3\} $ and let $\Omega \subset \R^n$ be an open, connected, bounded set with smooth boundary $\p \Omega$, let $T > 0$, and let $u$ be the solution of 
\bel{pf}
\begin{aligned}
\begin{cases}
\Box u:=\partial_t^2 u - \triangle u = f, 
&\text{on $(0,T) \times \Omega$},
\\
u =0,
&\text{on $(0,T) \times \p \Omega$},
\\
u|_{t=0} = u_0,\ \p_t u|_{t=0} = u_1
&\text{on $\Omega$}.
\end{cases}
    \end{aligned}
\ee
The initial data $u_0, u_1$ are assumed to be a priori unknown functions, but the source $f$ is assumed to be known, together with the additional piece of information 
\bel{data_omega}
  \begin{aligned}
q=u|_{(0,T)\times \omega},
    \end{aligned}
\ee
where $\omega \subset \overline{\Omega}$ is open.
The data assimilation problem then reads:\\

\noindent (DA) Find $u_0$ and $u_1$ given $f$ and $q$.\\

\noindent In typical applications $f = 0$. Due to the finite speed of propagation, $T$ needs to be large enough in order for (DA) to have unique solution. 
Assuming that 
\bel{uniqueness_T}
    \begin{aligned}
T > 2 \max \{\dis(x,\omega) \I x \in \overline \Omega \},
    \end{aligned}
\ee
it follows from Holmgren's unique continuation theorem that (DA) is uniquely solvable. Here $\dis(x,\omega) = \min\{\dis(x,y) \I y \in \overline \omega\}$ 
and $\dis(x,y)$ is the distance function in $\Omega$, defined as the infimum over the lengths of continuous paths in $\Omega$, joining $x$ and $y$.

The problem (DA) can be exponentially ill-posed under the assumption (\ref{uniqueness_T}). In order to avoid such severely ill-posed cases, we will suppose that the geometric control condition holds in the sense of \cite{LLTT}. This means roughly speaking that any billiard trajectory intersects $\omega$ before time $T$. A billiard trajectory leaving from a point in $\Omega$ consists of line segments that are joined together at points on $\p \Omega$, with directions satisfying Snell's law of reflection.
However, the exact formulation of the geometric control condition requires also a consideration of trajectories gliding along $\p \Omega$.
It is well-known that the geometric control condition characterizes the cases where the problem (DA) is stable, and in a slightly different context, the characterization originates from \cite{BLR}.

We will analyse the convergence of a finite element method that gives an approximate solution to (DA). Our method is based on piecewise affine elements in space and the use of backward finite differences in time. The main contribution of the paper is to show that, when
complemented with a suitable stabilization, even this standard, low
order discretization, yields a convergence that is optimal with
respect to approximation and the stability of the continuous
problem. The stabilization terms are carefully designed balancing the numerical
stability, the approximation properties of the scheme and the stability of the continuous
problem. This allows us to prove linear convergence with respect to the mesh size in
the global space time $L^2$-norm, reflecting the Lipschitz stability
of the continuous problem. This stability holds under the assumption
that the continuum problem satisfies the geometric control condition.
The analysis also considers the effect of discretizing the smooth domain, as
well as other perturbations of the data. The resulting scheme is on
the form of a time-space primal-dual system. The forward
equation is independent of the dual. Therefore the gradient can be
computed by a forward solve, followed by a dual backward solve,
for steepest descent type iterative solving.
 
We hope that the present paper can act as a starting point for
exploration of more applied, but also more advanced, stabilized finite
element methods. Indeed although stabilization terms herein are
taylored for the low order method, the approach is general and can be
extended to other finite element methods. For instance, it might be desirable to use high order elements
in space and a more sophisticated discretization in time in order to reduce the numerical dissipation. 

\subsection{Previous literature}

There are two extensive traditions of research that are closely related to the problem (DA). As already mentioned above, a variation of (DA) arises as a mathematical model for the medical imaging technique called photoacoustic tomography (PAT), and works related to PAT form one of the two traditions. We refer to \cite{PAT_review, Xu2006, Wang2009} for physical aspects of PAT, and to \cite{Kuchment2008,scherzer2010handbook} for mathematical reviews.

The problem (DA) models wave propagation in a cavity $\Omega$, whereas the classical PAT problem is formulated in $\R^3$. However, the papers \cite{Acosta2015, lCO, Nguyen2015, Stefanov2015b} study the PAT problem in a cavity. All these papers consider methods based on using iterative time reversal for the continuum wave equation, an approach that originates from \cite{Stefanov2009a}, and none of them consider the issues arising from discretization. 

The second tradition draws from control theory, and it uses so-called Luenberger observers. The data assimilation problem (DA) arises as the dual problem of a control problem, and analysis of the latter is typically reduced to the analysis of (DA) by using the Hilbert uniqueness method originating from \cite{Lions1988}.

A Luenberger observers based algorithm was first analysed in a finite dimensional ODE context in \cite{Auroux2005}. An abstract version of the method, applicable to the problem (DA), was introduced in \cite{Ramdani2010}.
The two traditions have a significant overlap. For instance, as pointed out in \cite{lCO}, 
the result \cite{Nguyen2015} on the PAT problem fits in the abstract setting of \cite{Ramdani2010}.
In particular, the methods in both the traditions can be formulated as Neumann series in infinite dimensional spaces. 

The paper \cite{lHR} studies a discretization of a Luenberger observers based algorithm. The error estimate in \cite{lHR} depends linearly on the point of truncation of the Neumann series (see Theorem 1 there), and this ultimately leads to logarithmic convergence with respect to the mesh size. 
The issue with the truncation can be avoided if a stability estimate is available on a scale of discrete spaces. Such estimates were first derived in \cite{IZ} and we refer the reader to
the survey articles \cite{Z,EZ}, as well as the recent paper \cite{EMZ} for more
details. However, quoting \cite{Cindea2015}, such estimates are proven only in ``specific and somehow academic situations''.
We refer to the monograph \cite{Ervedoza2013}, see in particular Chapter 5 on open problems, for a detailed discussion of the truncation issue in the context of the control problem, dual to (DA).

The closest work to the present paper is \cite{Cindea2015}. There two finite element methods for (DA) are considered: one of them is stabilized while the other is not. The method without stabilization is shown to converge only under the further assumption that certain discrete inf-sup condition holds, see (42) there. 
On the other hand, the stabilized method is shown to converge to the exact solution only under a further regularity assumption on an auxiliary Lagrangian multiplier, see $\lambda$ in Proposition 2 there. 
Under this assumption, it is then shown in the $1+1$-dimensional case, that the stabilized method converges with quadratic rate when the Bogner-Fox-Schmit $C^1$-elements, with third order polynomials, are used in spacetime rectangles. 

The data assimilation problem (DA) can also be solved using the
quasi-reversibility method. This method originates from \cite{lLL}, and it has been applied to data assimilation problems subject to the wave equation in \cite{KM, KR}, and more recently to the PAT problem in \cite{lCK}. Another interesting application is given in the recent preprint \cite{lBDP}. There the authors solve an obstacle detection problem by using a level set method together with the quasi-reversibility method applied to a variant of (DA).

The quasi-reversibility method introduces an auxiliary Tikhonov type regularization parameter. When deriving a rate of convergence for the method, this parameter needs to be chosen as a function of the mesh size $h$. In  \cite{lCK} the regularization parameter is called $\epsilon$, and by balancing the estimates in Theorems 3.3, 4.6 and 5.3 there, we are lead to the choice $\epsilon(h) = h^{2/3}$. This gives the convergence rate $h^{2/3}$ for the quasi-reversibility method \cite{lCK}.

To summarize, the linear convergence rate of our method is superior to that of the Neumann series based methods and the quasi-reversibility method. Contrary to \cite{Cindea2015} it is also optimal with respect to the order of the finite elements used. The convergence proof is based on using the continuum estimates, and the only geometric assumption needed is the sharp geometric control condition. Finally, the method uses a very simple discretization of the spacetime, and it is likely that the ideas presented here can be adapted to various other discretizations. 

Let us also mention that the method in the present paper draws from our experience on stabilized finite element methods for the elliptic Cauchy problem \cite{Bu13,Bu14}, and other types of data assimilation problems, see \cite{BON} for elliptic and \cite{BO,BIO} for parabolic cases. In \cite{BON} we considered the Helmholtz equation. The convergence estimate there is explicit in the wave number, and exhibits a hyperbolic character in the sense that it relies on a convexity assumption that can viewed as a particular local version of the geometric control condition.





\section{Continuum Estimates}

The main aim of this section is to recall a continuum observability estimate for the wave operator under some geometric assumptions on the observable domain $\cO=(0,T) \times \omega$. In order to state these geometric conditions we will need the following definition. We refer the reader to \cite{LLTT} for the definition of compressed generalized bicharacteristics.

\begin{defn}[See \cite{BLR},\cite{LLTT}]
We say that $\cO\subset\M$ satisfies the geometric control condition in $\M$, if every compressed generalized bicharacteristic $^b\gamma(s)=(t(s),x(s),\tau(s),\xi(s))$ intersects the set $\cO$ for some $s \in \R$.
\end{defn}

\noindent With this definition in mind, we can state the continuum estimate that is used to derive a convergence rate for our finite element method:

\begin{thm} 
\label{continuum}
Suppose $\mathcal{M}= (0,T) \times \Omega$ where $\Omega$ is a domain with smooth boundary. Let $\omega \subset \overline{\Omega}$ and assume that $\cO=(0,T)\times \omega$ satisfies the geometric control condition. If $u \in L^2(\mathcal{M})$ with $u(0,\cdot) \in L^2(\Omega)$, $\partial_t u (0,\cdot) \in H^{-1}(\Omega)$, $u|_{(0,T) \times \partial \Omega}=h \in L^2((0,T)\times \partial \Omega)$ and $\Box u = f \in H^{-1}(\mathcal{M})$, then $u \in C^1([0,T];H^{-1}(\Omega))\cap C([0,T];L^2(\Omega))$ and
$$\sup_{t \in [0,T]}(\|u(t,\cdot)\|_{L^2(\Omega)}+\|\p_t u(t,\cdot)\|_{H^{-1}(\Omega)})  \lesssim \|u\|_{L^2(\cO)} + \|f\|_{H^{-1}(\mathcal{M})}+\|h\|_{L^2((0,T) \times \partial \Omega)}.$$
\end{thm}

\noindent Theorem ~\ref{continuum} is a consequence of the following homogeneous version:

\begin{thm}[Observability estimate]
\label{internal obs}
Let $\cO$ satisfy the geometric control condition. There exists a constant $C>0$ such that for any initial data $w|_{t=0}=g_1\in L^2(\Omega)$ and $\partial_t w|_{t=0}=g_2\in H^{-1}(\Omega)$, the corresponding unique weak solution $w$ to $\Box w=0$, $w|_{(0,T) \times \partial \Omega}=0$ with
$$ w \in C((0,T);L^2(\Omega)) \cap C^1((0,T);H^{-1}(\Omega))$$
satisfies:
$$ \|g_1\|_{L^2(\Omega)} + \|g_2\|_{H^{-1}(\Omega)} \leq C \|w\|_{L^2(\cO)}.$$
\end{thm}

Theorem ~\ref{internal obs} is a classical result that yields an interior observability estimate under the geometric control condition. The proof of the theorem uses propagation of singularities for the wave equation and only works for smooth geometries. The geometric control condition is essentially a necessary and sufficient condition for obtaining the observability estimate and roughly states that all light rays in $\mathcal{M}$ must intersect $\cO$ taking into account reflections at the boundary \cite{BLR}. We refer the reader to \cite[Proposition 1.2]{LLTT} for a proof of this theorem using a combination of the study of semiclassical defect measures and propagation of singularities. One can also look at \cite[Theorem 3.3]{BLR} for an alternative proof using propagation of singularites. The paper \cite{BLR} deals with boundary observability but the proof can be applied to obtain interior observability as well. We omit rewriting these proofs here as they are well known in the literature. Let us remark at this point that there is a stronger geometric condition on the observable domain $\cO$ known as the $\Gamma-$ condition which is much simpler to verify in general. We recall the $\Gamma-$condition defined as follows
\begin{defn}
For each $ x_0 \notin \Omega$, Let $\Gamma_{x_0}:= \{ x \in \partial \Omega \, | \,(x-x_0)\cdot \nu(x) \geq 0 \}$. We say that $\cO=(0,T)\times \omega$ satisfies the $\Gamma-$condition if 
$$ \exists x_0 \notin \Omega, \quad \exists \delta>0 \quad \text{such that} \quad \mathcal{N}_\delta(\Gamma_{x_0}) \cap \Omega \subset \omega,$$
$$ T > 2 \sup_{x \in \Omega} |x-x_0|,$$
where $\mathcal{N}_\delta(\Gamma_{x_0}):= \{ y \in \R^n \, | \, |y-x| < \delta \quad \text{for some} \quad x \in \Gamma_{x_0} \}.$
\end{defn}

\noindent It is known that the $\Gamma-$ condition implies the geometric control condition (see for example \cite{M}). In essence, the $\Gamma-$condition roughly requires $T$ and $\bar{\omega} \cap \overline{\partial\Omega}$ to be relatively large. Although not as sharp as the geometric control condition, the advantage of the $\Gamma-$ condition lies in its applicability in the presence of non-smooth geometries and the explicit derivation of the constant $C$ in Theorem ~\ref{internal obs}. For an alternative proof of Theorem ~\ref{internal obs} in the case that $\cO$ satisfies the $\Gamma-$condition, we refer the reader to \cite[Theorem 2.2]{DZZ}. One can also use the Carleman estimate \cite[Theorem 1.1]{BBE} to derive this estimate although in this case one has to shift the Sobolev estimates.

A key ingredient in deriving the Lipschitz stability result in this paper is a corollary of the observability estimate for the wave equation as stated in Theorem ~\ref{continuum}. In the remainder of this section, we will show that Theorem ~\ref{continuum} indeed follows from the observability estimate. To this end, we will need the following lemma concerning solutions to the mixed Dirichlet-Cauchy problem for the wave equation with weak Sobolev norms. We refer the reader to \cite[Theorem 2.3]{LLT} together with Remark 2.8 in that paper for the proof.
\begin{lem}
\label{solutions low regularity}
Let $\Omega$ be a bounded domain with smooth boundary. Suppose $(u_0,u_1,f,h) \in X$ where $X= L^2(\Omega) \times H^{-1}(\Omega)\times H^{-1}(\M)\times L^2((0,T)\times\partial\Omega)$ with the usual product topology. Then the equation (\ref{pf}) has a unique solution $u \in Y:=C^1([0,T];H^{-1}(\Omega)) \cap C([0,T];L^2(\Omega))$. Furthermore, the linear mapping that maps $(u_0,u_1,f,h)$ to $u$ is continuous:
\[ \|u\|_{Y} \lesssim \|(u_0,u_1,f,h)\|_X.\]
\end{lem}
\noindent We are now ready to show the derivation of Theorem ~\ref{continuum} from Theorem ~\ref{internal obs}.
\begin{proof}[Proof of Theorem ~\ref{continuum}]
\noindent Let us consider the vector valued function $v:=[v_1\ v_2]^T$ with $v_i \in L^2(\mathcal{M})$ for $i \in \{1,2\}$ defined as the solution to the following separable system of PDEs:
\[
\left\{ \begin{array}{rcll} &\Box v = [f\ 0]^T\\
& v(t,x)=[h\ 0]^T \quad \forall x \in \partial \Omega, \forall t \in [0,T]\\
& v(0,x)=[0\ u_0]^T \quad \forall x \in \Omega\\
 &\partial_t v(0,x)=[0\ u_1]^T \quad \forall x \in \Omega. \end{array}\right.
\]
\noindent Note that if $w:=u-(v_1+v_2)$, then $w \in L^2(\mathcal{M})$ and $w$ satisfies the homogeneous wave equation
\[
\left\{ \begin{array}{rcll} &\Box w = 0\\
& w(t,x)=0 \quad \forall x \in \partial \Omega, \forall t \in [0,T]\\
& w(0,x)=0, \quad \forall x \in \Omega\\
 &\partial_t w(0,x)=0 \quad \forall x \in \Omega. \end{array}\right.
\]
\noindent \noindent By Lemma ~\ref{solutions low regularity}, we have $w=0$, which implies that $u=v_1+v_2$. Since $\cO$ satisfies the geometric control condition, the observability estimate in Theorem ~\ref{internal obs} holds for the function $v_2$ and together with Lemma ~\ref{solutions low regularity} we have that for all $t \in [0,T]$:
$$ \|v_2(t,\cdot)\|_{L^2(\Omega)}+\|\p_t v_2(t,\cdot)\|_{H^{-1}(\Omega)} \lesssim \|v_2\|_{L^2(\cO)}.$$
\noindent Similarly, applying Lemma ~\ref{solutions low regularity} to the function $v_1$ implies that:
$$ \|v_1(t,\cdot)\|_{L^2(\Omega)}+\|\p_t v_1(t,\cdot)\|_{H^{-1}(\Omega)}  \lesssim \|f\|_{H^{-1}(\mathcal{M})}+\|h\|_{L^2((0,T)\times\partial\Omega)}.$$
\noindent Finally, combining the above estimates, we deduce that:
\[\|u(t,\cdot)\|_{L^2(\Omega)}+\|\p_t u(t,\cdot)\|_{H^{-1}(\Omega)} \leq\|v_1(t,\cdot)\|_{L^2(\Omega)}+\|\p_t v_1(t,\cdot)\|_{H^{-1}(\Omega)}+\|v_2(t,\cdot)\|_{L^2(\Omega)}+\|\p_t v_2(t,\cdot)\|_{H^{-1}(\Omega)}\]
\[\lesssim \|v_1\|_{L^2(\cO)}+\|f\|_{H^{-1}(\mathcal{M})}+\|h\|_{L^2((0,T)\times \partial\Omega)}\lesssim\|f\|_{H^{-1}(\mathcal{M})}+\|h\|_{L^2((0,T)\times \partial\Omega)}+ \|u-v_1\|_{L^2(\cO)}\]
\[\lesssim \|f\|_{H^{-1}(\mathcal{M})}+\|h\|_{L^2((0,T)\times \partial\Omega)}+ \|u\|_{L^2(\cO)}.\]
\end{proof}

\section{Discretization}
Let us begin with a brief discussion of the overall discretization approach employed in this paper. We consider the wave equation \eqref{pf} and the preliminary Lagrangian functional
$$ \mathcal{L}_0(u,z) = \frac{1}{2}\|u-q\|_{L^2((0,T)\times \omega)}^2 + \int_{\mathcal{M}} (\partial^2_t u) \,  z + \nabla u \cdot \nabla z - fz \, dxdt.$$ 
The Euler-Lagrange equations for $\mathcal{L}_0$ can be written as follows
\begin{align*}
\langle \partial_u \mathcal{L}_0(u,z),v \rangle &= \int_{0}^T\int_{\omega} (u-q)v \,dtdx + \int_{\mathcal{M}} (\partial^2_t v) \,  z + \nabla v \cdot \nabla z  \, dtdx=0,\\
\langle \partial_z \mathcal{L}_0(u,z),w \rangle &= \int_{\mathcal{M}} (\partial^2_t u) \,  w + \nabla u \cdot \nabla w - fw \, dtdx=0
\end{align*}
for all $v,w$. It is clear that if $u$ is equal to the unique solution to \eqref{pf} and $z \equiv 0$, then these Euler-Lagrange equations are satisfied. This simple idea outlines the overall approach in this paper. We will employ a discrete Lagrangian functional whose critical points will converge to the unique solution to the continuum problem. However, as the term  $\int_{0}^T\int_{\omega} (u-q)v \,dtdx$ does not seem to give enough stability for the discrete problem to converge, we will add certain regularization terms in the discrete setting. The design of these terms is driven by numerical stability and the stability of the continuous problem, with the objective of minimizing the computational error. In the final section of the paper we will briefly discuss the possibility of removing some of these regularization terms.\\

Let us now present the discretization of \eqref{pf}. We will first consider
a family of polyhedral domains $\Omega_h$ approximating  $\Omega$ and
similarly let
$\omega_h$ denote a family of domains approximating $\omega$.
Let $\mathcal{T}_h$ be a conforming triangulation of the polyhedral domain $\Omega_h$. Let $h_K=diam(K)$ be the local mesh parameter and $h=max_{K \in \mathcal{T}_h} h_K$ the mesh size. We assume that the family of triangulations $\mathcal{T}_h$ is quasi uniform. Let $V_h$ be the standard space of piecewise affine continuous finite elements satisfying the zero boundary condition,
$$V_h = \{v \in H^1_0(\Omega_h); v|_{K}\in\mathbb{P}_1(K), \forall K
\in \mathcal{T}_h\}.$$
We assume that the approximate geometries $\Omega_h$ and $\omega_h$ are
sufficiently close to $\Omega$ and $\omega$ in the following sense,
\bel{ext}
\dis(x,\partial \Xi) \lesssim h^2 \quad \forall x \in  \partial \Xi_h,
\quad \Xi= \Omega \mbox{ or } \Xi = \omega.
\ee
This is possible for domains
$\Omega,\, \omega$ with smooth boundary (see for example \cite{BK94}). We have the following lemma:
\begin{lem}(See \cite[Lemma 2]{BK94})
\label{extension}
Let the condition \eqref{ext} be satisfied. Then for all $v \in
H^1(\Omega \cup \Omega_h)$ the following estimate holds:
$$\int_{(\Omega \setminus \Omega_h)\cup (\Omega_h \setminus \Omega)}
|v(x)|^2 \,dx \lesssim h^2
\left(\int_{\partial \Omega}|v(x)|^2 \,ds+ h^2 \int_{\Omega}|\nabla v(x)|^2 \,dx\right).$$
\end{lem}
\noindent Following \cite{BIO} we first discretize in space only. To
take into account the mismatch between $\Omega_h$ and $\Omega$ we use
the stable extension operator \cite{stein1970}, $E:H^s(\mathcal{M}) \to H^s(\mathcal{M}_h)$, $s \ge 0$ with
$\mathcal{M}_h:= (\Omega\cup\Omega_h) \times (0,T)$ to
define the extended source function 
$
f^e =  E f, \quad f^e \vert_\Omega = f.
$
We may then write a semi-discrete finite element formulation of the problem as follows. Find $u \in C^2(0,T;V_h)$ such that 
$$ (\partial^2_{t} u , v)_h +a_h(u,v)=(f^e,v)_h, \hspace{5mm} \forall v \in V_h,$$
where
$$(u,v)_h= \int_{\Omega_h} uv \,dx ,\hspace{5mm} a_h(u,v)=\int_{\Omega_h} \nabla u \cdot \nabla v \,dx.$$
We also define
$$(u,v)_{\Omega}= \int_{\Omega} uv \,dx ,\hspace{5mm} a(u,v)=\int_{\Omega} \nabla u \cdot \nabla v \,dx.$$
\noindent Let $ N \in \mathbb{N}$ and $\tau>0$ satisfy $N \tau = T$ and define $t_n=n\tau$. Furthermore, define for each discrete function $u =(u^n)_{n=0}^N \in V_h^{N+1}$,
$$\partial_{\tau} u^n = \frac{u^n - u^{n-1}}{\tau}\quad \text{for} \quad n\in\{1,\ldots,N\} \quad \quad \partial^2_{\tau} u^n = \frac{u^n - 2u^{n-1}+u^{n-2}}{\tau^2}\quad \text{for} \quad n\in\{2,\ldots,N\}.$$
\noindent It is natural to assume that the two
discretization scales $\tau$ and $h$ should be comparable in size. We
will therefore assume throughout the paper that $\tau =
\mathcal{O}(h)$. To allow for a discrete set $\omega_h$ we assume
that data $q^n$ are known in the possibly (slightly) larger domain $\omega \cup \omega_h$.\\
\\
Consider the Lagrangian functional $\mathcal{L}: V_h^{N+1} \times V_{h}^{N-1} \to \mathbb{R}$ defined by:
\begin{equation} \label{Lagrangian}
\begin{split}
\mathcal{L}(u,z) & = \frac{\tau}{2} \sum_{n=1}^N \|u^n-q^n\|_{\omega_h}^2+G(u,z) -\tau \sum_{n=2}^N (f^n,z^n)_h\\ 
&+ \frac{1}{2} \|h\nabla u^1\|_h^2+\frac{1}{2} \|h\partial_{\tau} u^1\|_h^2+\frac{1}{2}\|h\nabla \partial_\tau u^1\|_h^2+\frac{1}{2}\|h\nabla \partial_\tau u^N\|_h^2+\frac{\tau}{2} \sum_{n=2}^N \|\tau \nabla \partial_{\tau}u^n\|_h^2, \\
G(u,z)&= \tau\sum_{n=2}^N(( \partial_{\tau}^2 u^n,z^n)_h+a_h(u^n,z^n)),
\end{split}
\end{equation}
for fixed functions $f \in C(0,T;L^2(\Omega_h))$ and $q \in
C(0,T;L^2(\omega_h))$,
$$f^n = f^e(t_n), \hspace{5mm} q^n = q(t_n), \hspace{5mm} n=1,...,N.$$ 
Define the bilinear forms $A_1$ and $A_2$ as follows:
\begin{equation} \label{bilinear}
\begin{split}
A_1(u,w) & = G(u,w), \\ 
 A_2((u,z),v) &= \tau \sum_{n=1}^N (u^n,v^n)_{\omega_h}+G(v,z) +
 (h\nabla u^1,h\nabla v^1)_h+(h\partial_{\tau} u^1,h\partial_{\tau}v^1)_h 
\\
&+(h\nabla \partial_\tau u^N, h \nabla \partial_\tau v^N)_h+(h\nabla \partial_\tau u^1, h \nabla \partial_\tau v^1)_h+\tau \sum_{n=2}^N (\tau \nabla \partial_{\tau}u^n, \tau \nabla\partial_{\tau}v^n)_h.
\end{split}
\end{equation}
Then the Euler-Lagrange equations for $\mathcal{L}$ are equivalent to:
\begin{equation} \label{EL}
A_1(u,w) =\tau \sum_2^N (f^n,w^n)_h \quad \text{and} \quad  A_2((u,z),v) = \tau \sum_1^N (q^n,v^n)_{\omega_h}.
\end{equation}

\subsection{Coercivity}
We let $z^0=z^1=z^{N+1}=z^{N+2}=0$ and define the following norms and seminorms:
\begin{equation} \label{norms}
\begin{split}
|||u|||_R^2 &=\tau \sum_{n=1}^N \|u^n\|_{\omega_h}^2 + \|h\nabla u^1\|_h^2+\|h\partial_{\tau}u^1\|_h^2   \\
&+\|h\nabla \partial_\tau u^1\|_h^2+\|h\nabla \partial_\tau u^N\|_h^2+\tau \sum_{n=2}^N \|\tau \nabla \partial_\tau u^n\|_h^2,\\ 
|||u|||_F^2 &= \tau \sum_{n=2}^N (\|\partial_{\tau}^2 u^n\|_h^2+\|\partial_\tau u^n\|_h^2 + \|\nabla u^n\|_h^2) + \|\nabla u^N\|_h^2 + \| \partial_\tau u^N\|_h^2,\\
|||z|||_D^2 &=  \frac{T}{2} \|z^N\|_h^2+ \frac{\tau}{4} \sum_{n=2}^N \|z^n\|_h^2 + \frac{\tau}{2(T+1)^2}\sum_{n=2}^N \|\nabla \tilde{z}^n\|_h^2+ \frac{1}{4(T+1)}\|\nabla \tilde{z}^N\|_h^2,\\
|||(u,z)|||^2_C &=|||u|||_R^2 + \tau \sum_{n=2}^N \|z^n\|_h^2.\\  
\end{split}
\end{equation}
Here $\tilde{z}^n := \tau \sum_{m=0}^n (1+m\tau) z^m$. Note that using the Poincar\'{e} inequality we have the following:
$$ \|\nabla \tilde{z}^n\|_h \geq C \|\tilde{z}^n\|_h \hspace{5mm} n=1,2,...,N.$$
We  also note that $||| (\cdot,\cdot)|||_C$ is a norm on $V_h^{2N}$.
\begin{prop}[Coercivity estimate]
\label{estimate}
For all $ N \in \mathbb{N}$, $h>0$ and $(u,z) \in V_h^{2N}$ there exists $(v,w) \in V_h^{2N}$ such that:
$$|||u|||_R^2 +h^2|||u|||_F^2 + |||z|||_D^2 \lesssim A_1(u,w) + A_2((u,z),v),$$
$$ |||(v,w)|||_C \lesssim |||u|||_R +  h |||u|||_F + |||z|||_D.$$
\end{prop}
\noindent Before proving this proposition, let us state a few lemmas. The first lemma is trivial.
\begin{lem}
$$ A_1(u,-z) + A_2((u,z),u) = |||u|||_R^2.$$
\end{lem}
\begin{rmk}
Observe that the form $A_2$ can be reduced by dropping the term
$(h\partial_{\tau} u^1,h\partial_{\tau}v^1)_h$ without sacrificing
stability since the contribution from this term is controlled by
$\|h \partial_{\tau} \nabla u^1\|^2_h$ by a Poincar\'e inequality.
\end{rmk}
\begin{lem}
\label{energy1}
Let $u \in V_h^{N+1}$. For $ n=2,3,...,N$ define:
$$ w^n :=\partial^2_{\tau} u^n+ (2T-n\tau) \partial_\tau u^n.$$
Then:
$$A_1(u, h^2 w) \geq C_1 h^2 |||u|||_F^2 - C_0 |||u|||_R^2,$$
where $C_0,C_1$ are constants that are independent of the parameter $h$. \\
\end{lem}

\begin{proof}
Recall that $A_1(u,h^2w) = h^2 G(u,w)$. Now, given the choice of the test function $w$ we have 
$$ G(u,w)=S_1+S_2+S_3+S_4,$$
where:
$$S_1 = \tau \sum_{n=2}^N \|\partial^2_{\tau} u^n\|_h^2, $$
$$S_2=\tau \sum_{n=2}^N (\partial^2_{\tau}u^n, (2T-n\tau) \partial_\tau u^n)_h$$
$$=\tau \sum_{n=2}^N (2T-n\tau) (\partial_\tau v^n, v^n)_h= \sum_{n=2}^N (2T-n\tau) \frac{1}{2}(\|v^n\|_h^2-\|v^{n-1}\|_h^2+\|v^n-v^{n-1}\|_h^2)$$
$$ \geq \frac{T}{4}\|\partial_\tau u_N\|_h^2+\frac{\tau}{2} \sum_{n=1}^N \|\partial_\tau u^n\|_h^2 - T \|\partial_\tau u^1\|_h^2 + \frac{T}{2}\tau^2 \sum_{n=2}^N \|\partial^2_\tau u^n\|_h^2,$$
$$S_3=\tau \sum_{n=2}^N a_h(u^n, (2T-n\tau) \partial_\tau u^n)= \sum_{n=2}^N (2T-n\tau) a_h(u^n, u^n-u^{n-1})$$
$$=  \sum_{n=2}^N (2T-n\tau)\frac{1}{2}( a_h(u^n, u^n)-a_h(u^{n-1},u^{n-1})+ a_h(u^n-u^{n-1},u^n-u^{n-1}))$$
$$ \geq \frac{T}{4}\|\nabla u^N\|_h^2 +\frac{\tau}{2} \sum_2^N \|\nabla u^n\|_h^2 - T \|\nabla u^1\|_h^2 + \frac{\tau^2}{2} \sum_{n=2}^N \|\tau \nabla \partial_\tau u^n\|_h^2,$$
$$S_4=\tau \sum_{n=2}^N a_h(u^n, \partial^2_\tau u^n) $$
$$=-\tau \sum_{n=2}^N (\nabla \partial_\tau u^{n-1}, \nabla \partial_\tau u^{n})_h- (\partial_\tau \nabla u^1, \nabla u^1)_h + (\nabla u^N, \partial_\tau \nabla u^N)_h.$$
Hence:
$$|S_4| \leq \tau \sum_{n=1}^N \| \nabla \partial_\tau u^n\|_h^2+\frac{1}{2}\|\partial_\tau \nabla u^1\|_h^2+\frac{1}{2}\|\nabla u^1\|_h^2+\frac{\delta}{2}\|\nabla u^N\|_h^2+\frac{1}{2\delta}\|\partial_\tau \nabla u^N\|_h^2.$$
\noindent One can see that by combining the above estimates the claim follows immediately for $\delta$ sufficiently small.
\end{proof}
\begin{lem}
\label{energy2}
Let $z \in V_h^{N-1}$. For $n=0,1,...,N$ define:
$$ v^n = \tau \sum_{m=0}^n (1 + m \tau) z^m := \tilde{z}^n.$$
Then:
$$ G(v,z) \geq |||z|||_D^2.$$
\end{lem}
\begin{proof}
\noindent Note that:
\begin{align*}
\tau \sum_{n=2}^N ( \partial^2_{\tau} v^n ,z^n)_h &=\tau \sum_{n=2}^N (\partial_\tau ((1+n\tau)z^n),z^n)_h= \tau \sum_{n=2}^N (z^n,z^{n-1})_h + \sum_{n=2}^N (1+n\tau)(z^n-z^{n-1},z^n)_h\\
&= \tau \sum_{n=2}^N \|z^n\|_h^2-\tau^2 \sum_{n=2}^n (z^n,\partial_\tau z^n)_h + \frac{1}{2}\sum_{n=2}^N (1+n\tau)(\|z^n\|_h^2-\|z^{n-1}\|_h^2+\|z^n-z^{n-1}\|_h^2)\\ 
&\geq \frac{\tau}{2} \sum_{n=2}^N \|z^n\|_h^2 + \frac{T}{2} \|z^N\|_h^2 -\tau^2 \sum_{n=2}^n (z^n,\partial_\tau z^n)_h + \frac{\tau^2}{2} \sum_{n=2}^N \|\partial_\tau z^n\|_h^2\\
&\geq \frac{\tau}{2} \sum_{n=2}^N \|z^n\|_h^2 + \frac{T}{2} \|z^N\|_h^2 -\tau^2 \sum_{n=2}^n \|z^n\|_h^2 - \frac{\tau^2}{4} \sum_{n=2}^N \|\partial_\tau z^n\|_h^2 + \frac{\tau^2}{2} \sum_{n=2}^N \|\partial_\tau z^n\|_h^2\\
&\geq \frac{\tau}{4} \sum_{n=2}^N \|z^n\|_h^2 + \frac{T}{2} \|z^N\|_h^2+ \frac{\tau^2}{4} \sum_{n=2}^N \|\partial_\tau z^n\|_h^2.
\end{align*}
\noindent Similarly: 
\begin{align*}
\tau \sum_2^N a_h(v^n,z^n)&= \tau \sum_{n=2}^N \frac{1}{(1+n\tau)} a_h(v^n,\partial_\tau v^n)= \sum_{n=2}^N\frac{1}{(1+n\tau)} a_h(v^n,v^n-v^{n-1})\\
&= \frac{1}{2} \sum_{n=2}^N \frac{1}{1+n\tau}(a_h(v^n,v^n)-a_h(v^{n-1},v^{n-1})+a_h(v^n-v^{n-1},v^{n}-v^{n-1}))\\
&\geq  \frac{1}{2} \sum_{n=2}^N \frac{1}{1+n\tau}(a_h(v^n,v^n)-a_h(v^{n-1},v^{n-1}))\\
&\geq \frac{1}{2} \sum_{n=2}^N \frac{1}{1+n\tau}a_h(v^n,v^n) - \frac{1}{2} \sum_{n=2}^{N-1} \frac{1}{1+n\tau}a_h(v^n,v^n)\\
& + \frac{\tau}{2} \sum_{n=2}^N \frac{1}{(1+n\tau)(1+(n-1)\tau)}a_h(v^{n-1},v^{n-1})\\
& \geq \frac{1}{4(1+T)}a_h(v^N,v^N) + \frac{\tau}{2(1+T)^2} \sum_{n=2}^N a_h(v^{n},v^{n}).\\
\end{align*}
\noindent Combining the above inequalities yields the claim.
\end{proof}

\begin{proof}[Proof of Proposition ~\ref{estimate}]
\noindent Let $\alpha$ be a sufficiently small parameter independent of $h$ and let us define
$$ \hat{v} = u + \alpha v, \quad \text{and} \quad \hat{w}= -z+ h^2\alpha w,$$
where $w,v$ are chosen as in Lemma ~\ref{energy1} and Lemma ~\ref{energy2} respectively. We will show that the claim holds for this specific choice of $(\hat{v},\hat{w}) \in V_h^{2N}$. Indeed,
\begin{equation} 
\begin{split}
A_1(u,\hat{w})+ A_2((u,z),\hat{v}) &= |||u|||_R^2 + \alpha A_1 (u,h^2 w) + \alpha A_2((u,z),v)\\
&\geq |||u|||_R^2 + \alpha C_1 h^2 |||u|||_F^2 - \alpha C_0 |||u|||_R^2+\alpha A_2((u,z),v)\\
& \geq  \frac{1}{2}( |||u|||_R^2 + \alpha C_1 h^2 |||u|||_F^2)+\alpha A_2((u,z),v).
\end{split}
\end{equation}
Now recalling that $v^0=v^1=0$, we see that:
\begin{align*}
&A_2((u,z),v)= G(v,z) +  \tau \sum_{n=1}^N (u^n,v^n)_{\omega_h}+(h\nabla \partial_\tau u^N, h \nabla \partial_\tau v^N)_h+\tau \sum_{n=2}^N (\tau \nabla \partial_{\tau}u^n, \tau \nabla\partial_{\tau}v^n)_h\\
&\geq  |||z|||_D^2 +   \tau \sum_{n=1}^N (u^n,v^n)_{\omega_h}+(h\nabla \partial_\tau u^N, h \nabla \partial_\tau v^N)_h+\tau \sum_{n=2}^N (\tau \nabla \partial_{\tau}u^n, \tau \nabla\partial_{\tau}v^n)_h.
\end{align*}
\noindent We have:
$$  \sum_{n=2}^N(\tau \nabla \partial_\tau u^n, \tau \nabla \partial_\tau v^n)_h \leq \sum_{n=2}^N (\frac{1}{2\delta}\|\tau \nabla \partial_\tau u^n\|_h^2 + C\frac{\delta}{2}\|\nabla \tilde{z}^n\|_h^2).$$

\noindent Similarly using the Cauchy-Schwarz inequality we have:
$$ \sum_{n=1}^N (u^n,v^n)_{\omega_h} \leq \sum_{n=1}^N  (\frac{1}{2\delta} \|u^n\|_{\omega_h}^2 + \frac{\delta}{2}\|v^n\|_h^2).$$
$$|(h \nabla \partial_\tau u^N,h\nabla \partial_\tau v^N)_h| \leq \frac{1}{2\delta}\|h\nabla \partial_\tau u^N\|_h^2 + 4T^2\frac{\delta}{2}\|z^N\|_h^2.$$
\noindent One can easily see that the estimate holds for $\alpha$ small enough and $\delta$ sufficiently smaller than $\alpha$. A similar argument yields the following estimate:
$$ |||(\hat{v},\hat{w})|||_C \lesssim |||u|||_R + C h |||u|||_F + C|||z|||_D.$$
Indeed we have:
\begin{align*}
&|||(\hat{v},\hat{w})|||_C^2 \leq  2(\tau \sum_{n=2}^N (\|u^n\|_\omega^2+\alpha^2\|v^n\|_\omega^2) + \|h\nabla u^1\|_h^2+\|h\partial_{\tau}u^1\|_h^2\\
&+\tau \sum_{n=2}^N (\|z^n\|_h^2+\alpha^2 h^4 \|w^n\|_h^2) +\tau \sum_{n=2}^N (\|\tau \nabla \partial_\tau u^n\|_h^2+\alpha^2\|\tau \nabla \partial_\tau v^n\|_h^2)\\
&+ \|h\nabla \partial_\tau u^1\|_h^2+ \|h\nabla \partial_\tau u^N\|_h^2+\alpha^2\|h\nabla \partial_\tau v^N\|_h^2).
\end{align*}
Note that:
$$\tau \sum_{n=2}^N \|v^n\|_\omega^2 \leq C |||z|||_D^2,$$ 
$$\tau \sum_{n=2}^N h^4\|w^n\|_h^2 \leq C h^2 |||u|||_F^2,$$
$$\tau \sum_{n=2}^N \|\tau \nabla \partial_\tau v^n\|_h^2 \leq C |||z|||_D^2,$$
$$\|h\nabla \partial_\tau v^N\|_h^2 \leq C \|z^N\|_h^2.$$
\end{proof}

One can use Proposition ~\ref{estimate} to show that the system of linear equations \eqref{EL} has a unique solution. Indeed, denote by $N_h$ the dimension of $V_h$. The equations \eqref{EL} define a square linear system with $2N_h \times N$ unknowns. Setting $f^n=q^n=0$, the coercivity estimate in Proposition ~\ref{estimate} implies that the kernel of this linear system is trivial and therefore there exists a unique solution for all choices of $f^n,q^n$. Henceforth, we will let $(u_h,z_h)$ denote the unique solution to \eqref{EL} subject to the measured data $f^n,q^n$. Next section is concerned with proving the convergence of the discrete solution $u_h$ to the continuum solution $u$ of \eqref{pf}. The dual variable $z_h$ is shown to converge to zero.\\

\section{A Priori Error Estimates}
An important feature of the error estimates below is that they include
bounds of the perturbations from the discretization of the domain. To
obtain such bounds we first prove some preliminary results. 
\begin{lem}\label{lem:geom_trace_error}
For all $v_h \in V_h$ there holds
\[
\|v_h\|_{\partial \Omega} \lesssim h \|\nabla v_h\|_{\Omega_h \setminus \Omega}.
\]
\end{lem}
\begin{proof}
First, note that using the trivial extension $v_h \vert_{\Omega\setminus\Omega_h}=0$ there holds $\|v_h\|_{\partial \Omega}=\|v_h\|_{\partial \Omega \cap \Omega_h}$. Now, for $x \in \partial \Omega \cap \Omega_h$, we write $v_h(x) = \int_{p(x)}^x \nabla
v_h \cdot
n\,ds$, with $n$ the outward pointing unit normal of $\partial
\Omega$, and where $p(x) := x + \zeta(x) n(x)$, with $\zeta$ is the (signed) distance from $\partial
\Omega$ to $\partial \Omega_h$ in the $n$ direction.  By the
assumption \eqref{ext}, $|\zeta| \lesssim h^2$ and there holds
\begin{equation}\label{eq:CS_bound}
\int_{p(x)}^x \nabla v_h \cdot
n\,ds \leq |\zeta(x)|^{\frac12} \left(\int_{p(x)}^x |\nabla v_h \cdot
n|^2\,ds\right)^{\frac12} \lesssim h \left(\int_{p(x)}^x |\nabla v_h \cdot
n|^2\,ds\right)^{\frac12}.
\end{equation}
Using the above expression for $v_h\vert_{\partial \Omega}$ we have
that
\[
\|v_h\|_{\partial \Omega}^2 \lesssim h^2 \int_{\partial \Omega} \int_{p(x)}^x |\nabla v_h \cdot
n|^2\,ds\,dx \lesssim h^2 \|\nabla v_h\|_{\Omega_h \setminus \Omega}^2.
\]
\end{proof}

First we define an $H^1$-projection $\pi_h : H^1_0(\Omega) \to
V_h(\Omega_h)$. Given $u\in H^1_0(\Omega)$, we let $\pi_h u \in V_h$ to be the unique solution of
\begin{equation}\label{eq:H1proj}
a_h(\pi_h u,v_h) = a_h(E u,v_h),\quad \forall v_h \in V_h
\end{equation}
\begin{lem}\label{lem:approx_H1}
Let $u \in H^1_0(\Omega)$ and let $\pi_h u \in V_h$ be
defined by \eqref{eq:H1proj}. Then:
\begin{equation}\label{eq:L2ritzbound}
\|u - \pi_h u\|_\Omega \lesssim h |u|_{H^{1}(\Omega)}
\end{equation}
and moreover
\begin{equation}\label{eq:H1ritzbound}
\|E u - \pi_h u\|_{H^{1}(\Omega_h )} \lesssim h |u|_{H^{2}(\Omega)} \mbox{ for } u \in H^1_0(\Omega) \cap H^2(\Omega).
\end{equation}
\end{lem}
\begin{proof}
First consider \eqref{eq:H1ritzbound}. Let $i_h u \in V_h$ denote the nodal interpolant of $E u$. 
By the
Poincar\'e's inequality there holds
\[
\|i_h u - \pi_h u\|_{H^1(\Omega_h)}^2 \lesssim a_h(i_h u - \pi_h u,i_h u - \pi_h u).
\]
Using the definition of $\pi_h u$, equation \eqref{eq:H1proj}, we have
\[
a_h(i_h u - \pi_h u,i_h u - \pi_h u) = a_h(i_h u -  E u,i_h u - \pi_h u)
\leq   \|i_h u -  E u\|_{H^1(\Omega_h)}\|i_h u - \pi_h
u\|_{H^1(\Omega_h)}. 
\]
Combining the above estimate with 
\[
\|i_hu-E u\|_{H^1(\Omega_h)} \lesssim h \|u\|_{H^2(\Omega_h)} \lesssim h \|u\|_{H^2(\Omega)}
\]
Dividing with $\|i_h u - \pi_h u\|_{H^1(\Omega_h)}$ and using this
estimate, it follows that
\[
\|i_h u - \pi_h u\|_{H^1(\Omega_h)} \lesssim   h |u|_{H^2(\Omega)}.
\]
The inequality \eqref{eq:H1ritzbound} follows by the triangle
inequality. For \eqref{eq:L2ritzbound}, first extend $\pi_h u$ to $\Omega$ by defining
$\pi_h u = 0$ in $\Omega \setminus \Omega_h$. Then define the dual problem
\[
\begin{array}{rcl}
-\Delta z &=& u - \pi_h u \quad \mbox{ in } \Omega\\
z&=&0 \quad \mbox{ on } \partial \Omega.
\end{array}
\]
By the smoothness of $\Omega$ we know that $|z|_{H^2(\Omega)} \lesssim
     \|u - \pi_h u\|_{\Omega}$.
It follows that
\[
\|u - \pi_h u\|_{\Omega}^2 = (u - \pi_h u,-\Delta z)_\Omega = (\nabla
     (u - \pi_h u), \nabla z)_\Omega+(\pi_h u, \nabla z \cdot
     n)_{\partial \Omega \cap \Omega_h}.
\]
For the first term in the right hand side we have
\[
(\nabla
     (u - \pi_h u), \nabla z)_\Omega = (\nabla
     (u - \pi_h u), \nabla E z)_{\Omega_h}- (\nabla
     (u - \pi_h u), \nabla E z)_{\Omega_h\setminus \Omega} +  (\nabla
     (u - \pi_h u), \nabla z)_{\Omega\setminus \Omega_h}.
\]
Therefore, recalling that by trivial extension, $\pi_h u\vert_{\Omega\setminus\Omega_h} =
0$,
\begin{multline}
\|u - \pi_h u\|_{\Omega}^2 = (u - \pi_h u,-\Delta z)_\Omega = (\nabla
     (u - \pi_h u), \nabla (E z - i_hE z))_{\Omega_h} \\
- (\nabla
     (u - \pi_h u), \nabla E z)_{\Omega_h\setminus \Omega} +  (\nabla
     (u - \pi_h u), \nabla z)_{\Omega\setminus \Omega_h}+(\pi_h u, \nabla z \cdot
     n)_{\partial \Omega \cap \Omega_h}\\
=  (\nabla
     (u - \pi_h u), \nabla (z - i_hz))_{\Omega_h}- (\nabla
     (u - \pi_h u), \nabla E z)_{\Omega_h\setminus \Omega}+ (\nabla u,\nabla
     z)_{\Omega\setminus\Omega_h}+(\pi_h u, \nabla z \cdot
     n)_{\partial \Omega \cap \Omega_h}\\
=I + II + III+IV.
\end{multline}
For the first term of the right hand side we have
\[
I \lesssim \|\nabla
     (u - \pi_h u)\|_{\Omega_h}  h |z|_{H^2(\Omega)}  \lesssim \|\nabla
     (u - \pi_h u)\|_{\Omega_h} h \|u - \pi_h u\|_\Omega \lesssim
     \|\nabla u\|_\Omega h  \|u - \pi_h u\|_\Omega,
\]
where we used that by \eqref{eq:H1proj} and the stability of the
extension operator there holds
\begin{equation}\label{eq:H1stabritz}
\|\nabla \pi_h u\|_{\Omega_h} \leq \|\nabla
     E u\|_{\Omega_h} \leq \|\nabla
     u\|_{\Omega}.
\end{equation}
To bound the second term we recall that by Lemma \ref{extension}, a trace
inequality and the stability of the extension and of $z$ there holds
\[
\|\nabla E z\|_{\Omega_h\setminus \Omega} \lesssim h \|u - u_h\|_\Omega.
\]
Hence, using once again the stability \eqref{eq:H1stabritz}
\[
II \leq \|\nabla
     (u - \pi_h u)\|_{\Omega_h\setminus \Omega} \|\nabla E
     z\|_{\Omega_h\setminus \Omega} \lesssim \|\nabla u\|_\Omega h  \|u - \pi_h u\|_\Omega.
\]
Similarly we obtain for the third term
\[
III \leq \|\nabla
     u\|_{\Omega\setminus \Omega_h} \|\nabla 
     z\|_{\Omega\setminus \Omega_h} \lesssim \|\nabla u\|_\Omega h  \|u - \pi_h u\|_\Omega.
\]
To estimate the fourth term, we use the Cauchy-Schwarz inequality,
Lemma \ref{lem:geom_trace_error}
and the trace inequality, followed by the stability estimate on $z$,
\begin{multline*}
IV = (\pi_h u,\nabla z\cdot n)_{\partial \Omega\cap \Omega_h} \leq
\|\pi_h u\|_{\partial \Omega\cap \Omega_h} \|\nabla z\|_{\partial \Omega\cap \Omega_h} 
\lesssim h 
\|\nabla \pi_h u\|_{\Omega_h \setminus \Omega} \|z\|_{H^2(\Omega)}
\lesssim h \|\nabla u\|_{\Omega} \|u
     - \pi_h u\|_{\Omega}.
\end{multline*}
Collecting the bounds for terms $I$-$IV$ we conclude.
\end{proof}

\begin{prop}
\label{apriori error}
Suppose $\Omega_h,\Omega$ are as before and that $u\in H^3(\mathcal{M})$. Let $(u_h,z_h)$ be the unique solution to the Euler-Lagrange equations \eqref{EL} with
$f= \Box u$ and $q = u|_{(0,T)\times \omega_h}$. Then:
\begin{equation}
\label{apriori error1}
|||u_h -\pi_h u |||_R + h|||u_h -\pi_h u|||_F +|||z_h|||_D \lesssim h  \|u\|_{H^3(\M)},
\end{equation}
where $\pi_h u$ is the orthogonal projection defined by equation \eqref{eq:H1proj}.
\end{prop}

\begin{proof}
First we recall that by the stability estimate of Proposition ~\ref{estimate}, there is $(v,w) \in V_h^{2N}$ satisfying:
\begin{equation}\label{eq:stab1}
|||u_h -\pi_h u|||_R^2 +h^2|||u_h -\pi_h u|||_F^2 + |||z_h|||_D^2
\lesssim ( A_1(u_h -\pi_h u,w) + A_2((u_h -\pi_h u,z_h),v))
\end{equation}
and
\begin{equation}\label{eq:stab2}
|||(v,w)|||_C \lesssim |||u_h -\pi_h u|||_R + h |||u_h -\pi_h u|||_F
+ |||z_h|||_D.
\end{equation}

\noindent We will now bound the two terms of the right hand side of \eqref{eq:stab1}. Note that if $u^n= u(t_n)$ then:
$$ (\partial^2_t u^n, \psi)_\Omega + a(u^n,\psi) = (f^n ,\psi)_\Omega \hspace{5mm}
\forall \psi \in H^1_0(\Omega).$$
Observe that given any $\phi \in H^1_0(\Omega_h)$,
\begin{equation*}
(Ef^n ,\phi)_h = (Ef^n - \Box E u^n,\phi)_h+(\Box E u^n,\phi)_h 
= (Ef^n - \Box E u^n,\phi)_{\Omega_h \setminus \Omega}+(\Box E u^n,\phi)_h.
\end{equation*}
This implies by the left equation of \eqref{EL}, that for all $w \in V_h$
\[
A_1(u_h,w) =\tau \sum_{n=2}^N\left( (\varsigma^n_E,w)_{\Omega_h \setminus
  \Omega}+(\partial^2_t u^n,w)_h+a_h(u^n,w)\right),
\]
where, with some abuse of notation we identify $u^n$ with $E u^n$
outside $\Omega$ and $\varsigma^n_E:=Ef^n - \Box E u^n$ denotes the geometry residual term.
Together with equation \eqref{EL} and \eqref{eq:H1proj}, this implies that:
$$A_1(u_h-\pi_h u,w) = \tau \sum_2^N (\partial^2_t u^n
- \partial^2_\tau u^n , w^n)_h + \tau \sum_2^N
((1-\pi_h)\partial^2_{\tau}u^n,w^n)_h +\tau \sum_2^N (\varsigma^n_E,w^n)_{\Omega_h \setminus
  \Omega}.$$
First we observe that by Lemma  \ref{extension},
\[
(\varsigma^n_E,w^n)_{\Omega_h \setminus
  \Omega} \leq \|\varsigma^n_E\|_{\Omega_h \setminus
  \Omega} \|w^n\|_{\Omega_h \setminus
  \Omega} \lesssim (\|f^n\|_\Omega+ \|\partial^2_t u^n\|_\Omega+\|u^n\|_{H^2(\Omega)}) h^2 \|\nabla w^n\|_{h}.
\]
Let:
$$I_1 = \tau \sum_2^N \|(1-\pi_h)\partial^2_{\tau}u^n\|_h^2,$$
$$ I_2 = \tau \sum_2^N \|\partial^2_t u^n - \partial^2_\tau
u^n\|_h^2,$$
$$ I_3 = \tau \sum_2^N  h^2 (\|f^n\|^2_\Omega+ \|\partial^2_t u^n\|^2_\Omega+\|u^n\|^2_{H^2(\Omega)}) \lesssim h^2(\|f\|_{H^1(0,T;L^2(\Omega))}+\|\partial^2_t u\|^2_{H^1(0,T;L^2(\Omega))}+\|u\|^2_{H^1(0,T;H^2(\Omega))}).$$
Then clearly we have:
$$ A_1(u_h-\pi_h u,w) \lesssim (I_1+ I_2+ I_3)^{\frac{1}{2}}
|||(0,w)|||_C.$$
Here we used that since $\tau = O(h)$ and by a (discrete) Poincar\'e inequality 
$$
\tau \sum_2^N h^2 \|\nabla w^n\|^2_{h} \lesssim
\|h \nabla w^1\|_h^2 + \tau \sum_{n=2}^N \|\tau \nabla \partial_{\tau}
w^n\|_h^2 \leq |||(0,w)|||_C^2.
$$
It remains to bound $I_1$ and $I_2$. To this end, observe that:
$$\partial^2_{\tau} u^n = \frac{1}{\tau^2} (  \int_{t_{n-2}}^{t_n} (t-t_{n-2})\partial^2_{t} u \,dt - 2 \int_{t_{n-1}}^{t_n} (t-t_{n-1})\partial^2_{t}u\, dt).$$
Hence:
$$I_1 \leq \frac{1}{\tau^2} \sum_2^N ( 2 \int_{t_{n-2}}^{t_n}(t-t_{n-2})^2\|(\pi_h-1)\partial^2_{t}u\|_h^2 \,dt +8 \int_{t_{n-1}}^{t_n} (t-t_{n-1})^2\|(\pi_h-1)\partial^2_{t}u\|_h^2 \,dt)$$
$$ \leq  \sum_2^N ( 2 \int_{t_{n-2}}^{t_n}\|(\pi_h-1)\partial^2_{t}u\|_h^2 \,dt +8 \int_{t_{n-1}}^{t_n} \|(\pi_h-1)\partial^2_{t}u\|_h^2 \,dt)$$
$$ \lesssim h^2 \int_0^T \|\nabla \partial^2_{t}u\|_h^2 \,dt.$$
\noindent Similarly we have:
$$ \partial^2_{\tau}u^n - \partial^2_{t}u^n = \frac{1}{2 \tau^2}( - \int_{t_{n-2}}^{t_n} (t-t_{n-2})^2\partial^3_{t} u \,dt +2 \int_{t_{n-1}}^{t_n} (t-t_{n-1})^2\partial^3_{t}u\,dt).$$
\noindent Using this identity we obtain:
$$I_2 \leq \frac{1}{2\tau^3}( \sum_2^N  (\int_{t_{n-2}}^{t_n} (t-t_{n-2})^4 \,dt) (\int_{t_{n-2}}^{t_n}\|\partial^3_{t} u\|_h^2\,dt)  +4 \sum_2^N (\int_{t_{n-1}}^{t_n} (t-t_{n-1})^4\,dt)(\int_{t_{n-1}}^{t_n}\|\partial^3_{t}u\|_h^2\,dt))$$
$$ \leq C \tau^2 \int_0^T \|\partial^3_t u\|_h^2 dt.$$
Considering now the contribution from $A_2$, note that by definition
$$ A_2((u_h-\pi_h u,z_h),v) = \tau \sum_2^N (u^n - \pi_h^n u^n, v^n)_\omega - (h \partial_{\tau}\pi_hu^1,h\partial_\tau v^1)-(h\nabla\pi_hu^1,h\nabla v^1)$$ 
$$- \tau \sum_2^N(\tau \nabla \partial_\tau \pi_h u^n , \tau \nabla \partial_\tau v^n)- (h \partial_\tau \nabla \pi_h u^N, h \partial_\tau \nabla u^N) )- (h \partial_\tau \nabla \pi_h u^1, h \partial_\tau \nabla u^1) .$$
Hence:
$$A_2((u_h-\pi_hu,z_h),v) \leq C ( I_3+I_4+I_5+I_6+I_7+I_8)^{\frac{1}{2}} |||(v,0)|||_C,$$
where:
\begin{align*}
I_3 &= \tau  \sum_2^N \|u^n - \pi_h^n u^n\|_{\omega_h}^2 \lesssim h^2 \|\nabla u \|_{L^2(0,T;H^1(\Omega_h))}^2,\\
I_4&= \|h \partial_{\tau}\pi_hu^1\|_h^2 \lesssim h^2\|u \|_{H^2(0,T;L^2(\Omega_h))}^2,\\
I_5&=\|h\nabla\pi_hu^1\|_h^2\lesssim h^2\|\nabla u \|_{H^1(0,T;L^2(\Omega_h))}^2,\\
I_6 &= \tau \sum_2^N\|\tau \nabla \partial_\tau \pi_h u^n\|_h^2 \lesssim \tau^2 \int_0^T \|\nabla \partial_t u \|_h^2 \,dt,\\
I_7&= \|h \nabla \partial_\tau \pi_h u^N\|_h^2 \lesssim h^2\|\nabla u \|_{H^2(0,T;L^2(\Omega_h))}^2,\\
I_8&= \|h \nabla \partial_\tau \pi_h u^1\|_h^2 \lesssim h^2\|\nabla u \|_{H^2(0,T;L^2(\Omega_h))}^2.
\end{align*}
By the stability of the extension, all the norms over $\Omega_h$ can
now be bounded by norms of the same quantities over $\Omega$.
The claim follows by collecting the above bounds.
\end{proof}

\begin{cor}\label{cor:H1_apriori}
Under the same assumptions as in Proposition \ref{apriori error} there holds
\[
|||u_h - u |||_R + h|||u_h - u|||_F +|||z_h|||_D \lesssim h  \|u\|_{H^3(\M)},
\]
and
\[
\|\nabla u_h^0\| +|||u|||_R+ |||u_h|||_F \leq C \|u\|_{H^3(\M)}.
\]
\end{cor}

\begin{proof}
The first inequality is immediate by adding and subtracting $\pi_h u$,
applying the triangle inequality followed by Proposition \ref{apriori error} and Lemma \ref{lem:approx_H1} and similar Taylor expansion
arguments as in Proposition \ref{apriori
  error}.
In the second inequality we note that
\[
\|\nabla u_h^1\|_h^2 + \|\nabla u_h^0\|_h^2 \lesssim \|\nabla
u_h^1\|_h^2+\tau^2 \|\nabla \partial_\tau u_h^1\|_h^2.
\]
We then add and subtract $u$ in the right hand side of the last
inequality and in $|||u_h|||_F$ and proceed as before, using the first
inequality of the result to control the $u-u_h$ part and a Taylor
expansion argument for the second.
\end{proof}

Before presenting the main theorem of this section we need an additional definition and lemma as follows. For each $w \in H^1_0(\mathcal{M})$, let us introduce the time averaged function $(\bar{w}^n)_{n=1}^N$ through
\[
\bar w^n = \tau^{-1} \int_{t^{n-1}}^{t^n} w \,dt,
\]
and denote by $\bar w$ the piecewise constant function 
$\bar w\vert_{[t^{n-1},t^n]} = \bar w^n$.


\begin{lem}
\label{bt}
Suppose $u \in H^3(\M)$ is the unique solution to the continuum problem \eqref{pf} and let $(u_h,z_h)$ denote the discrete solution to the Euler-Lagrange equations \eqref{EL}. The following estimate holds:

$$ |(\nabla u_h^1,\nabla \bar w^1)| \leq C \| u\|_{H^3(\M)} \|w\|_{H^1(\mathcal{M})},$$
where $w \in H^1_0(\mathcal{M})$ is arbitrary and  $C$ is a constant independent of the parameter $h$.
\end{lem} 
\begin{proof}
\noindent Note that 
\bel{hh}
(\nabla u_h^1, \nabla \bar{w}^1) = -\tau (\nabla \partial_\tau u_h^2, \nabla \bar{w}^1)+(\nabla u_h^2, \nabla \bar{w}^1).
\ee
For the first term on the right hand side of equation \eqref{hh}, observe that:
$$|\tau (\nabla \partial_\tau u_h^2, \nabla \bar{w}^1)| \leq \tau \|\nabla \partial_\tau u_h^2\| \|\nabla \bar{w}^1\| \leq C \|u\|_{H^3(\M)} \|w\|_{H^1(\mathcal{M})}, $$
where we are using Corollary ~\ref{cor:H1_apriori} to bound
$\|\nabla \partial_\tau u_h^2\| \leq C \tau^{-\frac{1}{2}} \|u\|_{H^3(\M)}$
and the stability of $\bar w$ for the bound $\|\nabla \bar{w}^1\| \leq \tau^{-\frac{1}{2}} \|w\|_{H^1(\mathcal{M})}$. For the second term on the right hand side of equation \eqref{hh} we have
$$ (\nabla u_h^2,\nabla \bar{w}^1)=(\nabla u_h^2, \nabla \pi_h \bar{w}^1)=(f^2,\pi_h \bar{w}^1) - (\partial^2_\tau u_h^2, (\pi_h-1) \bar{w}^1)- (\partial^2_\tau u_h^2,  \bar{w}^1)=I +II + III.$$
We note that Theorem \ref{apriori error} implies that $\|\partial_\tau^2 u_h^2\| \leq C\tau^{-\frac{1}{2}}\|u\|_{H^3(\M)}$. Now
$$|I| \leq C \|f\|_{H^1((0,T):L^2(\Omega))}\|w\|_{H^1(\mathcal{M})} \leq C \|u\|_{H^3(\M)},$$
$$|II| \leq C \| \partial^2_\tau u_h^2\| \tau^{\frac{1}{2}}\|w\|_{H^1(\mathcal{M})}\leq C \|u\|_{H^3(\M)} \|w\|_{H^1(\mathcal{M})},$$
$$|III|  \leq C \| \partial^2_\tau u_h^2\| \|\bar{w}^1\| \leq C \|u\|_{H^3(\M)} \|w\|_{H^1(\mathcal{M})},$$

\noindent where in the last step we are using the standing assumption that $\tau \sim h$ and
\[
\|\bar w^1\| 
 \leq
\tau^{-\frac12} (\int_0^\tau \|w(t,\cdot)\|_h^2 \,dt)^{\frac12},
\]
\[
\int_0^\tau \int_\Omega |\int_0^t \partial_t w(t,\cdot) \, dt|^2\,dx
\,dt \leq \int_0^\tau \int_\Omega \tau \int_0^\tau |\partial_t w(t,\cdot)|^2 \, dt\,dx\,dt = \tau^2 \|\partial_t w\|^2_{L^2((0,t_1)\times\Omega)}.
\]

\end{proof}
\noindent We are now ready to state the main theorem as follows. 

\begin{thm}
\label{error estimate}
Suppose $\cO=(0,T)\times\omega$ satisfies the geometric control condition. Let $u \in H^3(\M)$ denote the unique solution to the continuum problem \eqref{pf} with $q \in L^2(\cO)$ and let $(u_h,z_h)$ denote the unique discrete solution to the Euler-Lagrange equations \eqref{EL}. Extend $u_h$ to all of $\Omega$ by setting it equal to zero in $\Omega \setminus \Omega_h$. The following error estimate holds:
$$\sup_{t\in[0,T]}(\|u(t,\cdot) - \tilde{u}_h(t,\cdot)\|_{L^2(\Omega)}+\|\p_t u(t,\cdot)-\p_t\tilde{u}_h(t,\cdot)\|_{H^{-1}(\Omega)}) \leq C h \|u\|_{H^3(\M)},$$
where $\tilde{u}_h \in C(0,T;L^2(\Omega))$ denotes the linear interpolation
$$ \tilde{u}_h = \frac{1}{\tau}( (t-t_{n-1}) u_h^n + (t_n-t)u_h^{n-1}) \quad \forall t \in [t_{n-1},t_n] .$$
\end{thm}

\begin{proof}
Recall the standing assumption that $\tau = \mathcal{O}(h)$. Let $e= u -\tilde{u}_h$ and define the linear functional
\bel{bounda}
\langle r,w\rangle = \int_0^T \int_{\Omega} (-\partial_t e \cdot \partial_t w + \nabla e \cdot \nabla w) \,dx \,dt    \hspace{5mm} \forall w \in H^1_0(\mathcal{M}). 
\ee
Applying Theorem \ref{continuum} we see that there holds
\[
\sup_{t\in[0,T]}(\|e(t,\cdot)\|_{L^2(\Omega)}+\|\p_t e(t,\cdot)\|_{H^{-1}(\Omega)}) \lesssim \| e\|_{L^2(\cO)} +
\|r\|_{H^{-1}(\M)} + \| e\|_{L^2((0,T)\times \partial \Omega)} .
\]
We will show that:
\begin{equation}
\label{boundary}
\| e\|_{L^2((0,T)\times \partial \Omega)} \leq C h \|u\|_{H^3(\M)},
\end{equation}
\begin{equation}
\label{local}
\| e\|_{L^2(\cO)} \leq C h \|u\|_{H^3(\M)},
\end{equation}
and
\begin{equation}
\label{global}
|\langle r,w\rangle| \lesssim h  \|u\|_{H^3(\M)} \|w\|_{H^1_0(\mathcal{M})}.
\end{equation}
Estimate \eqref{boundary} and  \eqref{local} will basically follow
once we control the $L^2$ norm of the error function in $(0,T) \times
\partial \Omega$ and $(0,T) \times \omega_h$, but \eqref{global} will
be more delicate as there is no immediate relation that bounds $\|\Box
e\|_{H^{-1}(\mathcal{M})}$ from above by $\|\Box
e\|_{H^{-1}((0,T)\times \Omega_h)}$. Let us begin with
\eqref{boundary}. Since $u(t) \in H^1_0(\Omega)$
\[
\| e\|^2_{L^2((0,T)\times \partial \Omega)} = \|\tilde
u_h\|^2_{L^2((0,T)\times \partial \Omega)} \lesssim \tau \sum_{n=0}^N \|u_h\|^2_{L^2((0,T)\times \partial \Omega)}.
\] 
Applying Lemma \ref{lem:geom_trace_error} followed by Corollary ~\ref{cor:H1_apriori} we have
\[
\tau \sum_{n=0}^N \|u^n_h\|^2_{L^2((0,T)\times \partial \Omega)}
\lesssim \tau \sum_{n=0}^N h^2 \|\nabla u^n_h\|_{h}^2 \lesssim
h^2\|\nabla u_h^0\|^2+h^2\|\nabla u_h^1\|^2+h^2 |||u_h|||_F^2 \lesssim C h^2 \|u\|_{H^3(\M)}^2.
\]
Now we consider the bounds \eqref{local} and \eqref{global}. Define the time discrete projection operator $\pi_0$ as follows:
$$ \pi_0 v := v(t^n) \hspace {5mm} \forall t \in (t_{n-1}, t_n], \hspace{5mm} n=1,...,N.$$
Then:
$$ \|\pi_0 v - v \|_{L^2(0,T)} \leq  \tau \| \partial_t v \|_{L^2(0,T)}.$$
We have:
$$ \|e\|_{L^2((0,T)\times\omega_h)}^2 \leq C(h^2 +\tau^2)\|u\|_{H^1(\mathcal{M})}^2 + \int_0^T \|\pi_0\pi_h u - \tilde{u}_h\|_{\omega_h}^2 \,dt,$$
and 
$$ \int_0^T \|\pi_0\pi_h u - \tilde{u}_h\|_{\omega_h}^2 \,dt \leq  \int_0^T \|\pi_0\pi_h u - \pi_0 \tilde{u}_h\|_{\omega_h}^2 \,dt+ \int_0^T \|\pi_0 \tilde{u}_h- \tilde{u}_h\|_{\omega_h}^2 \,dt$$
$$= \tau \sum_1^N\|\pi_h u^n - u_h^n\|_{\omega_h}^2+ \sum_1^N \int_{t_{n-1}}^{t_n} \|\pi_0 \tilde{u}_h- \tilde{u}_h\|_{\omega_h}^2 \,dt.$$
Here the first term is bounded by $|||u_h - \pi_h u |||_R$ and we use the identity
$$ \tilde{u}_h(t) = u_h^n + (t-t_n)\partial_{\tau} u^n_h, \quad t \in (t_{n-1},t_n]$$
to estimate the second one as follows:
$$\sum_1^N \int_{t_{n-1}}^{t_n} \|\pi_0 \tilde{u}_h- \tilde{u}_h\|_{\omega_h}^2 \,dt=\sum_1^N \int_{t_{n-1}}^{t_n} \|(t_n-t)\partial_\tau u^n_h\|_{\omega_h}^2 \,dt \leq \tau \sum_1^N \|\tau \partial_\tau u_h^n\|_h^2$$
$$ \leq   \tau \sum_1^N \|\tau \partial_\tau \pi_h u^n\|_h^2+ \tau \sum_1^N \|\tau \partial_\tau (\pi_h u^n - u_h^n)\|_h^2.$$
The first term above is bounded by $\tau^2\|u\|_{H^3(\M)}^2$ and as $\tau = \mathcal{O}(h)$, the second term is bounded by $h^2|||\pi_h u - u_h|||^2_F$. Hence, using  Proposition ~\ref{apriori error} we deduce that
$$\|e\|_{L^2((0,T)\times\omega_h)}^2 \lesssim h^2 \|u\|_{H^3(\M)}^2.$$
\noindent Now, using Lemma ~\ref{extension} on the domains $\omega_h$ and
$\omega$, by choosing $v=u$ and noting that $v(t) \in H^{1}(\Omega)$ for a.e $t \in (0,T)$ we obtain that:
\begin{multline}
 \|e\|_{L^2((0,T) \times (\omega \setminus \omega_h))}^2 \lesssim 
 h^2 (\|e\|_{L^2((0,T) \times \partial \omega)}^2+ h^2
  \|e\|^2_{L^2((0,T);H^1(\omega))})\\
\lesssim h^2 
  \|e\|^2_{L^2((0,T);H^1(\Omega))} \lesssim h^2 
  (\|u\|^2_{L^2((0,T);H^1(\Omega))}+ \|\tilde u_h\|^2_{L^2((0,T);H^1(\Omega))}).
\end{multline}
By the Poincar\'e inequality and the definition of $\tilde u_h$
\[
\|\tilde u_h\|^2_{L^2((0,T);H^1(\Omega))} \lesssim \tau \sum_{n=1}^N
 (\|\nabla u^n_h\|_h^2+\|\nabla u^{n-1}_h\|^2_{L^2(\Omega)} )
\lesssim \tau \sum_{n=0}^N \|\nabla u^n_h\|_h^2.
\]
We now observe that, using the second inequality of Corollary \ref{cor:H1_apriori}
\[
\tau \sum_{n=0}^N \|\nabla u^n_h\|_h^2 = \tau (\|\nabla
u^0_h\|_h^2+\|\nabla
u^1_h\|_h^2) +|||u_h|||_F^2 \lesssim \|u\|_{H^3(\M)}^2
\]

\noindent Finally, combining the preceding four inequalities yields the desired claim \eqref{local}. We now prove \eqref{global}. 
Using the definition of $e$ and the equation $\Box u =f$, we see that
\bel{boundres}
\langle r,w \rangle = \int_0^T \int_{\Omega} f w \,dx \,dt - \int_0^T
\int_{\Omega} (-\partial_t \tilde u_h \cdot \partial_t w + \nabla
\tilde u_h \cdot \nabla w) \,dx \,dt.
\ee
Recalling that $u_h$ has been extended by zero and that by extension
$w\vert_{\Omega_h \setminus \Omega} = 0$, we have
\bel{boundresh}
\langle r,w \rangle = \int_0^T \int_{\Omega_h} f^e w \,dx
\,dt+\int_0^T \int_{\Omega \setminus \Omega_h} f^e w \,dx \,dt - \int_0^T
\int_{\Omega_h} (-\partial_t \tilde u_h \cdot \partial_t w + \nabla
\tilde u_h \cdot \nabla w) \,dx \,dt 
\ee
\noindent Using integration by parts and recalling that $w(0,\cdot) = w(T,\cdot)=0$ we have
\[
\int_0^T
\int_{\Omega_h} (-\partial_t \tilde u_h \cdot \partial_t w)  \,dx \,dt
= 
 \tau \sum_{n=1}^{N-1} \int_{\Omega_h} \partial_\tau^2 u_h^{n+1} w(\cdot,t^n) \,dx
\]
Now, recalling the definition of the time averaged function $\bar w$, and considering the right hand side of \eqref{boundresh} we see that
\begin{multline}
\langle r,w \rangle = \underbrace{\int_0^T \int_{\Omega \setminus \Omega_h} f^e w \,dx + (f, (w - \bar w) )_{\mathcal{M}_h}}_{I} + \underbrace{(f,\bar
w)_{\mathcal{M}_h}}_{II} - \underbrace{\tau \sum_{n=1}^{N-1}
(\partial_\tau^2 u_h^{n+1}, w(\cdot,t^n)-\bar
w^{n+1})_h}_{III}\\
- \underbrace{\tau \sum_{n=2}^N  [(\partial_\tau^2 u_h^{n},\bar
w^{n})_h+(\nabla u_h^{n},\nabla \bar w^{n})_h]}_{IV} - \underbrace{\tau (\nabla u_h^{1},\nabla \bar w^{1})_h }_{V}
-\underbrace{\sum_{n=1}^N \int_{t_{n-1}}^{t_n}(t-t_n) (\nabla \partial_\tau u_h^n,\nabla w)_h \,dt }_{VI}.
\end{multline}
We now proceed to bound the six terms $I$-$VI$ of the right hand
side. First using Lemma \ref{extension} and a Poincar\'e inequality,
\[
I \leq C (\tau + h) \|f\|_{L^2(\mathcal{M})}(\|\nabla
w\|_{L^2(\mathcal{M})} + \|\partial_t w\|_{L^2(\mathcal{M})}).
\]
Then we use approximation in time on $f - f^1$ and $f - f^n$
\[
II = \int_{t^{0}}^{t^1} (f - f^1,\bar w^1)_h \,dt + \tau (f^1,\bar w^1)_h  +\sum_{n=2}^N \int_{t^{n-1}}^{t^n} (f - f^n,\bar w^n)_h \,dt +  \tau
 \sum_{n=2}^N (f^n,\bar w^n)_h\]
\[
\lesssim \tau (\|f\|_{H^1(\mathcal{M})}
\|w\|_{H^1(\mathcal{M})} ) + \underbrace{
\tau \sum_{n=2}^N (f^n,\bar w^n)_h}_{IV'}.
\]
Using that $\|w(\cdot, t)-\bar
w^{n+1}\|_{(t_n,t_{n+1})} \lesssim \tau \|\partial_t
w\|_{(t_n,t_{n+1})}$ we have for term $III$
\begin{multline}
III = \sum_{n=1}^{N-1}
(\partial_\tau^2 u_h^{n+1}, \int_{t_n}^{t_{n+1}} (\int_{t_n}^{t}
(\partial_s w(\cdot,s) \,ds +w(\cdot, t)-\bar
w^{n+1}) \,dt))_h \\ \lesssim \tau \left(\sum_{n=2}^N \tau\| \partial^2_\tau
u_h^n\|^2_h\right)^{\frac12} \|w\|_{H^1(\mathcal{M})}\lesssim \tau
|||u_h|||_F \|w\|_{H^1(\mathcal{M})}.
\end{multline}
For the sum of terms $IV'$ and $IV$ we use \eqref{EL} and
\eqref{eq:H1proj} to obtain
\begin{multline}
IV'+IV =  \tau \sum_{n=2}^N [(f^n,\bar w^n)_h-(\partial_\tau^2 u_h^{n},\bar
w^{n})_{h}-a(u_h^{n},\bar w^{n})]\\ 
= \tau \sum_{n=2}^N  [(f^n,\bar w^n-\pi_h \bar w^n)_h-(\partial_\tau^2 u_h^{n},\bar
w^{n}-\pi_h \bar w^n)_h]\\
\lesssim \tau |||u_h|||_F  \|\nabla w\|_{L^2(\mathcal{M})}
\lesssim h \|u\|_{H^3(\M)} \|w\|_{H^1(\mathcal{M})}.
\end{multline}

\noindent The estimate for $V$ follows immediately from Lemma
~\ref{bt}. Finally for the term $VI$ we use Cauchy-Schwarz inequality
and Corollary \ref{cor:H1_apriori} to write
\begin{equation}
VI \lesssim \left(\tau \sum_{n=1}^{N} \|\tau \nabla \partial_\tau u_h^{n}\|^2_h\right)^{\frac12}\|w\|_{H^1(\mathcal{M})}  
 \lesssim |||u_h|||_R \|w\|_{H^1(\mathcal{M})} \lesssim \tau\|u\|_{H^3(\M)} \|w\|_{H^1(\mathcal{M})}.
\end{equation}
\end{proof}
\section{Further Results}

\subsection{The case of perturbations in data}

Theorem ~\ref{error estimate} implies that the the Finite Element method described above is Lipschitz stable and therefore the extension of the above analysis to the case where the data is perturbed
is straightforward. Indeed assume that the observable data $(\hat{q},\hat{f})$ satisfies:
$$ \hat{q}= q + \delta q \quad \hat{f}=f + \delta f,$$

\noindent with $\delta q \in H^1((0,T);L^2(\omega))$ and $\delta f \in H^1((0,T);L^2(\Omega))$. Then a standard perturbation argument leads to similar results as in Proposition ~\ref{apriori error} and Theorem \ref{error estimate}, but with an additional term of
the form
\[ \|\delta q\|_{H^1(0,T;L^2(\omega))} + \|\delta f\|_{H^1(0,T;L^2(\Omega))}\]
\noindent in the right hand side of the bounds of the error estimates. This is a similar result
as one would obtain for a well-posed problem. Indeed this implies that the same Lipschitz stability estimate as in Theorem ~\ref{error estimate} would then hold as long as the perturbations in data are roughly the same as the size of the mesh parameter $h$. 

\subsection{Polyhedral boundaries}
\noindent Recall that the proof of the key continuum estimate in Theorem ~\ref{continuum} only works for smooth boundaries. Indeed the boundary smoothness assumption imposed in this paper is purely an artifact of the continuum estimate as the FEM would be much simpler to apply for polyhedral boundaries and the discrete solution $(u_h,z_h)$ would also exist and be unique. It is however possible to obtain a similar statement as in Theorem ~\ref{error estimate} for the case where $\Omega$ is a convex domain with a polyhedral boundary $\partial \Omega$. Here we present an admissibility condition that will be in some ways an alternative formulation of the geometric control condition or the $\Gamma-$condition for domains with polyhedral boundaries. Once this admissibility condition is satisfied for the observable domain $\cO$, one can proceed to prove that Theorem ~\ref{error estimate} holds. To formulate this condition, we assume that there exists an auxiliary exhaustion of the polyhedral domain $\Omega$ by a sequence $\{\Omega_n\}_{n \in \N}$ such that the following properties are satisfied:

\begin{itemize}
\item{$\forall n \in \N \quad \Omega_{n} \subset \Omega_{n+1}$,}
\item{$\mu(\Omega \setminus \Omega_n) \leq \frac{1}{n}$, where $\mu$ denotes the Lebesgue measure,}
\item{$\forall n \in \N \quad \partial \Omega_n \in C^{\infty}$.} 
\item{$(0,T)\times(\Omega_n \cap \omega) $ satisfies the geometric control condition in $(0,T) \times \Omega_n$ for all $n$.}
\item{The constants $C_n$ in the observability estimates corresponding to $(0,T) \times (\Omega_n \cap \omega)$ are uniformly bounded.}
\end{itemize} 
For polyhedral domains $\Omega$, we call the sets $\cO=(0,T)\times \omega$ with the above properties to be {\bf admissible}. Note that the first three conditions will always be possible for any polyhedral domain $\Omega$. It is merely the last two conditions which may not be true for an arbitrary domain $\cO$. It is easy to check that if $\cO$ satisfies the $\Gamma-$condition, then the admissibility condition above holds and therefore the implementation of the FEM in these cases works even for polyhedral boundaries $\Omega$. It would be a very interesting question to study how this admissibility condition can more generally be written for $\Omega, \omega, T$ without the use of the sequence $\Omega_n$. 

\subsection{Stability with less regularization}

Let us return to the explicit form of the Lagrangian functional $\mathcal{L}(u,z)$ in \eqref{Lagrangian} and sketch some heuristic arguments regarding the discrete level regularization terms and the possibility of altering or removing them. The terms $\frac{\tau}{2} \sum_2^N \|u^n-q^n\|_{\omega}^2+G(u,z) -\tau \sum_{2}^N (f^n,z^n)$ are absolutely necessary if we want the critical points of the Lagrangian functional to converge to the solution of \eqref{pf}. The two regularizer terms $\frac{1}{2} \|h\nabla u^1\|_h^2+\frac{1}{2} \|h\partial_{\tau} u^1\|_h^2$ control the initial energy of the system and seem to be a natural term in the regularization. However, the additional terms $\frac{1}{2}\|h\nabla \partial_\tau u^1\|_h^2+\frac{1}{2}\|h\nabla \partial_\tau u^N\|_h^2+\frac{\tau}{2} \sum_2^N \|\tau \nabla \partial_{\tau}u^n\|_h^2$ control a particular choice of mixed derivatives of $u$. The advantage of using these additional regularizer terms is that it yields Lipschitz stability of the FEM with the optimal rate $h$ and it avoids the use of any dual stabilizer terms for $z$ in the Lagrangian. There is some freedom in the selection of these regularizers. For example, one may be able to remove the bulk regularizer term $\frac{\tau}{2} \sum_2^N \|\tau \nabla \partial_{\tau}u^n\|_h^2$ and replace it with only initial and final data regularizers such as $\frac{1}{2}\|h\p^2_{\tau}u^2\|_h^2+\frac{1}{2}\|h\nabla \partial_\tau u^1\|_h^2+\frac{1}{2}\|h\nabla \partial_\tau u^N\|_h^2$ and obtain the same error estimate. This will require an alternative energy estimate (see Lemma ~\ref{energy1}) and as such will require the smoothness class $u \in H^4(\M)$.

One could prove Theorem ~\ref{error estimate} using a less number of regularization terms but at the cost of a slower rate of decay. For example, using the Lagrangian functional
$$  \hat{\mathcal{L}}(u,z)=\frac{\tau}{2} \sum_{n=2}^N \|u^n-q^n\|_{\omega}^2+G(u,z) -\tau \sum_{2}^N (f^n,z^n)+\frac{1}{2} \|h\nabla u^1\|_h^2+\frac{1}{2} \|h\partial_{\tau} u^1\|_h^2+\frac{1}{2}\|h\partial_{\tau}\nabla u^1\|_h^2,$$
it is possible to prove Theorem ~\ref{error estimate} with a slower rate of decay of $\mathcal{O}(\sqrt{h})$ for the error function. It is also possible to obtain a linear convergence for the error function in weaker norms using the following 'minimal' Lagrangian functional:
$$ \tilde{\mathcal{L}}(u,z)=\frac{\tau}{2} \sum_{n=2}^N \|u^n-q^n\|_{\omega}^2+G(u,z) -\tau \sum_{2}^N (f^n,z^n)+\frac{1}{2} \|h\nabla u^1\|_h^2+\frac{1}{2} \|h\partial_{\tau} u^1\|_h^2,$$
\noindent In this case, one can still prove Lemma ~\ref{energy2} in the exact same manner. A similar estimate can be proved for $u$ as well by choosing the test function $w$ through: 
$$ w^n :=(2T-n\tau) \partial_\tau u^n +  \tau \sum_{m=0}^n (1 + m \tau) u^m.$$
\noindent This will give positive control of $\|\tau \sum_{m=0}^n u^m\|_h^2$ together with $\tau \sum_{n=0}^N\|u^n\|_h^2$ and $\tau \sum_{n=1}^N\|\partial_\tau u^n\|_h^2$. Using these alternative estimates one can show that there exists a unique discrete solution $(u_h,z_h)$ to 
$$ \partial_{u} \tilde{\mathcal{L}}(u_h,z_h)=0,$$
$$ \partial_{z} \tilde{\mathcal{L}}(u_h,z_h)=0.$$

\noindent Now let $\mathcal{E}{u}^n = \tau \sum_{m=0}^n u^n$ and set
$$ \tilde{u}_h := \frac{1}{\tau}( (t-t_{n-1}) \mathcal{E}{u_h}^n + (t_n-t)\mathcal{E}{u_h}^{n-1}) \quad \forall t \in [t_{n-1},t_n] .$$

\noindent One can then prove that if $e := \int_0^t u - \tilde{u}_h$, then the following weak stability estimate for the above FEM holds as well:
$$ \|e\|_{L^2(\mathcal{M})} \leq C h \|u\|_{H^3(\M)}\quad \text{and}\quad\|u_0-u_h^0\|_{H^{-1}(\Omega)} \leq C h \|u\|_{H^3(\M)}.$$



\end{document}